\documentclass[11pt]{article}

\usepackage{fullpage}
\usepackage{arydshln}
\usepackage{latexsym}
\usepackage{amsfonts}
\usepackage{amssymb}
\usepackage{amsmath}
\usepackage{color}
\usepackage{graphicx, graphics}
\usepackage{hyperref}
\usepackage{pstricks, pst-plot, epsfig}
\usepackage{xargs}
\usepackage{amsthm}
\usepackage{enumerate}
\usepackage{subcaption}

\newcommand{\ud}{\,\mathrm{d}}
\renewcommand\Re{\operatorname{Re}}
\renewcommand\Im{\operatorname{Im}}

\newcommand{\CC}{\mathbb{C}}

\newcommand{\RR}{\mathbb{R}}

\newtheorem{prop}{Proposition}

\begin{document}

\title{The Linear KdV Equation with an Interface}
\author{Bernard Deconinck$^1$, Natalie E. Sheils$^2$, and David A. Smith$^{3}$ \\
\footnotesize 1. Department of Applied Mathematics, University of Washington, Seattle, WA 98195-3925 \\
\footnotesize 2. School of Mathematics, University of Minnesota, Minneapolis, MN 55455\\
\footnotesize 3. Division of Science, Yale-NUS College, 138527 Singapore \\
\footnotesize email\textup{: \texttt{nesheils@umn.edu}}
}
\date{\today}

\maketitle

\begin{abstract}
The interface problem for the linear Korteweg-de Vries (KdV) equation in one-dimensional piecewise homogeneous domains is examined by constructing an explicit solution in each domain. The location of the interface is known and a number of compatibility conditions at the boundary are imposed.  We provide an explicit characterization of sufficient interface conditions for the construction of a solution using Fokas's Unified Transform Method.  The problem and the method considered here extend that of earlier papers to problems with more than two spatial derivatives.\end{abstract}

\section{Introduction}
Interface problems for partial differential equations (PDEs) are initial boundary value problems for which the solution of an equation in one domain prescribes boundary conditions for the equations in adjacent domains. In applications, interface conditions are often obtained from conservation laws~\cite{Kevorkian}. Few interface problems allow for an explicit closed-form solution using classical solution methods. Using the Fokas or Unified Transform Method~\cite{DeconinckTrogdonVasan, FokasBook, FokasPelloni4}, such solutions may be constructed for both dissipative and dispersive linear interface problems as shown in~\cite{Asvestas, DeconinckPelloniSheils,Mantzavinos, SheilsDeconinck_PeriodicHeat, SheilsDeconinck_LS, SheilsDeconinck_LSp, SheilsSmith}. 

All previous papers addressing interface problems using the Fokas Method for interface problems have dealt exclusively with problems that are of second order in the spatial variable.  This paper is the first investigation into higher-order problems.  The process presented in this paper makes clear how to resolve new issues that arise when moving to a higher-order problem.

The nondimensionalized Korteweg-de Vries (KdV) equation $$q_t+6qq_x+q_{xxx}={}0,$$ is one of the most studied nonlinear PDEs~\cite{Hirota, Miura, MiuraGardnerKruskal, ZakharovFaddeev}.  It arises in the study of long waves in shallow water, ion-acoustic waves in plasmas, and in general, describes the slow evolution of long waves in dispersive media~\cite{AS}.  In this manuscript we study the linearized KdV equation (LKdV) in a composite medium,
\begin{equation}\label{LKdV_2i}
q_t={}\sigma(x)q_{xxx},~~~~~~-\infty<x<\infty,
\end{equation}
where $\sigma(x)$, a real-valued function, is piecewise constant.  This equation describes the behavior of solutions of the KdV equation in the small-amplitude limit and understanding its dynamics is fundamental in understanding the dynamics of the more complicated nonlinear problem.

In what follows an explicit solution method is given resulting in closed-form expressions.  We provide criteria which, under the additional assumption of existence of a solution, are sufficient for a solution representation to be obtained via the Fokas Method.  Although we do not prove uniqueness of the solution, we note some examples of interface conditions that do yield uniqueness.  The numerical evaluation of the solution is not considered but should be possible via the methods presented in~\cite{BiondiniTrogdon, Levin, TrogdonThesis, Trogdon}.  As we do not have a physical application on hand, this paper addresses the mathematical question of the number and type of interface conditions required to ensure that~\eqref{LKdV_2i} is well posed.

\section{Background}\label{sec:motivation}
Determining the number of boundary conditions necessary for a well-posed problem is a nontrivial issue, especially for boundary value problems (BVPs) with higher than second-order derivatives.  Consider LKdV posed on the half line
\begin{equation}\label{halflineKdV}
q_t={}\sigma^3 q_{xxx},~~~~x>0,~~~ t>0,
\end{equation}
where the form of the coefficient $\sigma^3$ is chosen for convenience.  If $\sigma<0$ then one boundary condition is needed, whereas if $\sigma>0$, two boundary conditions must be prescribed in order for the problem to be well posed~\cite{DeconinckTrogdonVasan, FokasBook}.  This difference in seemingly very similar BVPs is understood at an intuitive level by considering the phase velocity $c(k)={}-i\omega(k)/k$ where $\omega(k)={}i\sigma^3 k^3$~\cite{Kevorkian}.   Thus, the phase velocity is $c(k)={}\sigma^3 k^2.$  If $\sigma<0$ the phase velocity is negative and information travels toward the boundary as in Figure~\ref{fig:sink}.    If $\sigma>0$, the phase velocity is positive and information travels away from the boundary as in Figure~\ref{fig:source}.  Therefore, it seems reasonable that one must prescribe more boundary information.  Note that if we were solving~\eqref{halflineKdV} for $x<0$ these results would be switched.  This will become relevant in what follows for the interface problem on the whole line.

\begin{figure}[htbp]
\begin{center}
\begin{subfigure}[b]{.4\textwidth}
\centering
\def\svgwidth{2.5in}
   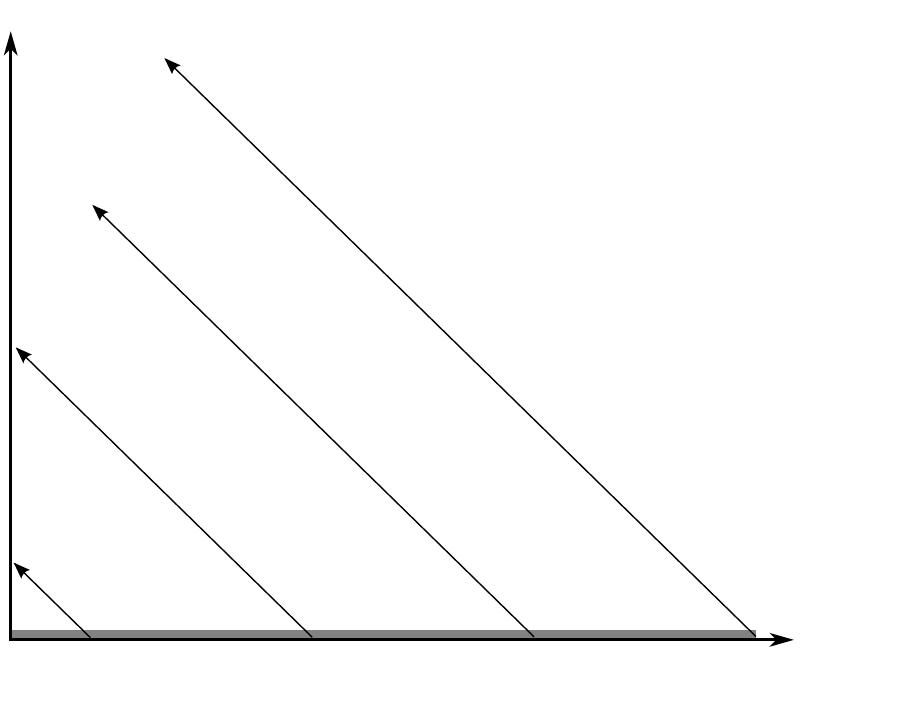 
   \caption{  \label{fig:sink}}
   \end{subfigure}
\begin{subfigure}[b]{.4\textwidth}
\centering
\def\svgwidth{2.5in}
   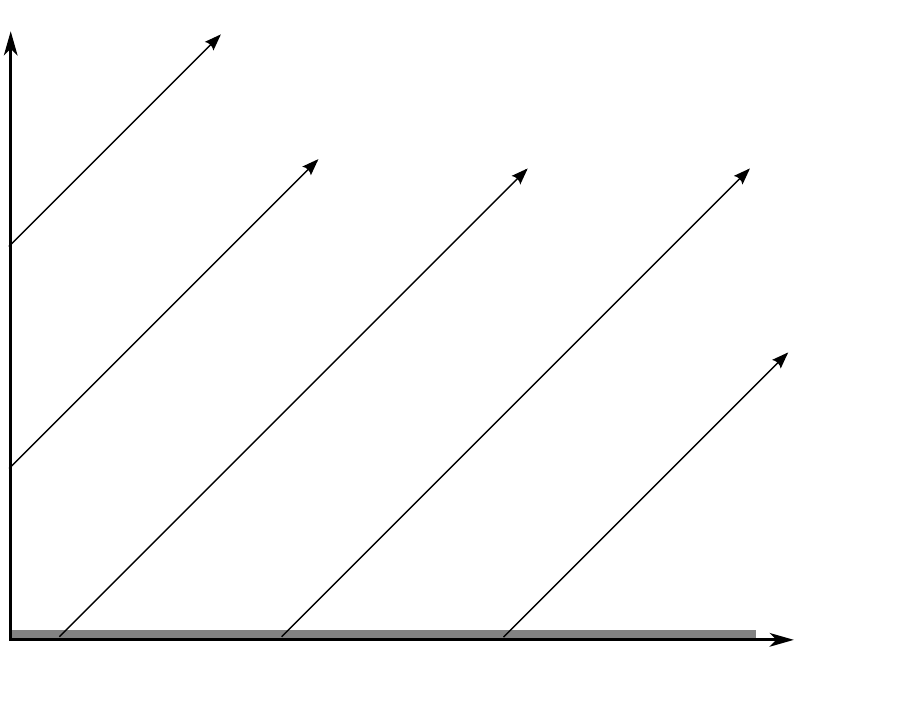 
   \caption{  \label{fig:source}}
   \end{subfigure}
   \caption{(a) When $\sigma<0$ information from the initial condition propagates toward the boundary $x={}0$ and one boundary condition needs to be prescribed. (b) When $\sigma>0$ information from the initial condition propagates away from the boundary and two boundary conditions need to be prescribed.}\label{fig:sinksource}
  \end{center}
\end{figure}

\vspace{.4in}
{\bf Remark.} The above argument hints at the Method of Characteristics.  This analogy is not justified, as the problem at hand is dispersive and energy spreads along phase or group velocity rays, rather than travels along them.  As such, the above argument is entirely heuristic but, with hindsight, it provides some intuition.
\vspace{.4in}

One of the strengths of the Fokas Method for solving linear PDEs is the straightforward way it enables determination of the number and type of boundary conditions required for a well-posed problem~\cite{DeconinckTrogdonVasan, FokasBook, FokasPelloni4}.  Previous papers by us and others~\cite{Asvestas, DeconinckPelloniSheils, Mantzavinos, SheilsDeconinck_PeriodicHeat, SheilsDeconinck_LS, SheilsDeconinck_LSp, SheilsSmith} have shown that the Fokas Method is useful for finding explicit general solutions to interface problems.  In the cases currently in the literature, only second-order problems are considered and the number of conditions required at each interface is clearly two.  The example of LKdV on the half-line suggests that the number of interface conditions needed in the case of LKdV with an interface depends on the sign of $\sigma$.  This is the case indeed.  In Propositions~\ref{prop:sigma_pn}--\ref{prop:sigma_np} we describe exactly the number and type of conditions necessary.

\section{Notation and Set-Up}
We investigate~\eqref{LKdV_2i} where $\sigma(x)$ is the piecewise constant real-valued function
	\begin{equation}
		\sigma(x) ={} \begin{cases} \sigma_1^3, & x<0, \\ \sigma_2^3, & x>0, \end{cases}
	\end{equation}
with the initial condition $q(x,0)={}q_0(x)$, and appropriate conditions at the interface $x={}0$.  The choice of the power $3$ in the definition of $\sigma(x)$ is purely for convenience.  We assume throughout this work that the solution decays rapidly to zero as $|x|\to\infty$. If nonzero conditions at $|x|={}\infty$ are desired this can be treated  easily in a manner similar to that for the heat equation in~\cite{DeconinckPelloniSheils} and for the linear Schr\"odinger equation in~\cite{SheilsDeconinck_LS}. We pose~\eqref{LKdV_2i} as the following interface problem:

\begin{subequations}\label{LKdVint_2i}
\begin{align}
\label{LKdV1_2i} q^{(1)}_t&={}\sigma_1^3 q^{(1)}_{xxx},& x<0, &&0<t\leq T,\\
\label{LKdV2_2i}  q^{(2)}_t&={}\sigma_2^3 q^{(2)}_{xxx},& x>0, &&0<t\leq T,
\end{align}
\end{subequations}
subject to the initial conditions

\begin{subequations}\label{LKdVic_2i}
\begin{align}
\label{LKdV1ic_2i} q^{(1)}(x,0)&={}q^{(1)}_0(x), &x<0,\\
\label{LKdV2ic_2i}  q^{(2)}(x,0)&={}q^{(2)}_0(x),& x>0,
\end{align}
\end{subequations}
with $q^{(1)}(\cdot,t)\in S(-\infty,0)$ and $q^{(2)}(\cdot,t)\in S(0,\infty)$ where $S(X)$ is the Schwartz space of restrictions to $X$ of rapidly decaying functions. Likewise, we assume rapid decay of the initial conditions, $q_0^{(1)}(\cdot)\in S(-\infty,0)$ and $q_0^{(2)}(\cdot)\in S(0,\infty)$. Note that we do not require $\hat{q}_0^{(1)}(0)={}\hat{q}_0^{(2)}(0)$, rather we assume that the initial data are compatible with the interface conditions, which are specified below.

Some number of interface conditions at $x={}0$ needs to be prescribed.  The number and type of such conditions are given in Propositions~\ref{prop:sigma_pn}--\ref{prop:sigma_np}.  We make a distinction in this manuscript between ``boundary problems" and ``interface problems."  Boundary problems are those in which the conditions given at the interface $x={}0$ allow one to solve either~\eqref{LKdV1_2i} or~\eqref{LKdV2_2i} as a half-line BVP without knowing the solution on the other domain.  For example, if one can solve a BVP for $q^{(1)}(x,t)$ then one can use that solution to provide any necessary conditions at $x={}0$ to solve the second BVP for $q^{(2)}(x,t)$.  Conditions for a well-posed BVP are given in~\cite{FokasBook, WangFokas}.  Since these cases have been examined, we restrict to those interface conditions which do not decouple such that either~\eqref{LKdV1_2i} or~\eqref{LKdV2_2i} can be solved as a BVP.  

It is of note that by making use of the PDE, interface conditions can always be written as
$$
	\left[\mbox{linear combination of }\frac{\partial^n}{\partial x^n} q^{(1)}(x,t)\bigg|_{x={}0} \mbox{ and } \frac{\partial^n}{\partial x^n}q^{(2)}(x,t)\bigg|_{x={}0} \mbox{ for } n={}0,1,2\right] ={} f(t),$$
for all $t$.  For example, one might require $q^{(1)}_{xxx}(0,t)={}q^{(2)}_{xxx}(0,t)$ as an interface condition.  This can be imposed by applying the equation and integrating in $t$ to give
\begin{equation}\label{3deriv_const}
\frac{1}{\sigma_1^3}q^{(1)}(0,t)-\frac{1}{\sigma_2^3}q^{(2)}(0,t)={}\frac{1}{\sigma_1^3}q^{(1)}_0(0)-\frac{1}{\sigma_2^3}q^{(2)}_0(0),
\end{equation} 
for all $t$, which is clearly of the form we require with $f(t)={}\frac{1}{\sigma_1^3}q^{(1)}_0(0)-\frac{1}{\sigma_2^3}q^{(2)}_0(0)$.  Using a similar process for any conditions on derivatives greater than second order as well as elementary linear algebra one can always express the interface conditions in the reduced forms given in Propositions~\ref{prop:sigma_pn}--\ref{prop:sigma_np} possibly after letting $x\to-x$.

\vspace{.4in}
{\bf Remark.}
If an interface condition specifies a linear combination of $\partial_x^nq^{(1)}(0,t)$ (n={}0,1,2) only or $\partial_x^nq^{(2)}(0,t)$ (n={}0,1,2) only, then we say it is a \emph{boundary condition}. Note that the interface conditions $$q^{(1)}(0,t)={}0, \qquad \mbox{and}\qquad q^{(1)}(0,t)-q^{(2)}(0,t)={}0,$$ are equivalent to the interface conditions $$q^{(1)}(0,t)={}0,\qquad\mbox{and}\qquad q^{(2)}(0,t)={}0,$$ so it is only meaningful to discuss the maximum number of boundary conditions for any equivalent expression of a given system of interface conditions. Henceforth any mention of a number of boundary conditions should be interpreted as such a maximum number of boundary conditions.

A problem with one boundary condition may or may not decouple into a pair of BVP.  Even if such a decoupling is possible, it may or may not be possible to solve the BVPs sequentially. For example, the problem with $\sigma_1,\sigma_2>0$, boundary condition $q^{(1)}(0,t)={}0$, and interface conditions $q^{(1)}_x(0,t)={}q^{(2)}_x(0,t)$ and $q^{(1)}_{xx}(0,t)={}q^{(2)}_{xx}(0,t)$ decouples into a solvable BVP for $q^{(1)}$ and a subsequent solvable BVP for $q^{(2)}$. However, the problem with $\sigma_1,\sigma_2<0$, and the same boundary and interface conditions does not decouple.
\vspace{.4in}

\section{Application of the Fokas Method}
We follow the standard steps in the application of the Fokas Method. Assuming existence of a solution, we begin with the so-called ``local relations'':
\begin{equation*}
(e^{-ikx+\omega_j t}q^{(j)})_t={}(e^{-ikx+\omega_j t}\sigma_j^3(q^{(j)}_{xx}+ikq^{(j)}_x-k^2q^{(j)}))_x, 
\end{equation*}
with $\omega_j={}\omega_j(k)={}i\sigma_j^3k^3$ for $j={}1,2$.  Applying Green's Theorem~\cite{AblowitzFokas}  and integrating over the strips $(-\infty,0)\times(0,t)$ and $(0,\infty)\times(0,t)$ respectively (see Figure~\ref{fig:GR_domain2i}) we have the global relations
\begin{equation*}
\begin{split}
\int_{-\infty}^0 e^{-ikx+\omega_1 t}q^{(1)}(x,t)\ud x&={}\int_{-\infty}^0 e^{-ikx}q^{(1)}_0(x)\ud x\\
&+\int_0^t e^{\omega_1 s}\sigma_1^3\left(q^{(1)}_{xx}(0,s)+ikq^{(1)}_x(0,s)-k^2 q^{(1)}(0,s)\right)\ud s,
\end{split}
\end{equation*}
\begin{equation*}
\begin{split}
\int_{0}^\infty e^{-ikx+\omega_2 t}q^{(2)}(x,t)\ud x&={}\int_{0}^\infty e^{-ikx}q^{(2)}_0(x)\ud x \\
&-\int_0^t  e^{\omega_2 s}\sigma_2^3\left(q^{(2)}_{xx}(0,s)+ikq^{(2)}_x(0,s)-k^2 q^{(2)}(0,s)\right)\ud s.
\end{split}
\end{equation*}

\begin{figure}[t]
\begin{center}
\def\svgwidth{4in}
   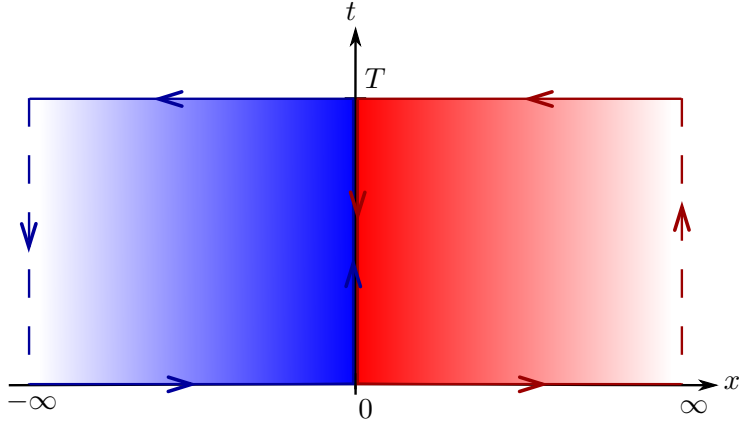 
   \caption{Regions for the application of Green's Formula in the case of two semi-infinite domains.   \label{fig:GR_domain2i}}
  \end{center}
\end{figure}

Let $\CC^+={}\{z\in\CC:\Im(z)\geq0\}$.  Similarly, let $\CC^-={}\{z\in\CC:\Im(z)\leq0\}$.  Define $k={}k_R+i k_I$ where $k_R,k_I\in\RR$.   We define the following:
\begin{align*}
&\hat{q}^{(1)}(k,t)={}\int_{-\infty}^0 e^{-ikx}q^{(1)}(x,t)\ud x, &k_I\geq0,&&0<t<T,\\
&\hat{q}^{(1)}_0(k)={}\int_{-\infty}^0 e^{-ikx}q^{(1)}_0(x)\ud x, &k_I\geq0,&&\\
&\hat{q}^{(2)}(k,t)={}\int_{0}^\infty e^{-ikx}q^{(2)}(x,t)\ud x, &k_I\leq0,& &0<t<T,\\
&\hat{q}^{(2)}_0(k)={}\int_{0}^\infty e^{-ikx}q^{(2)}_0(x)\ud x, &k_I\leq0,&&\\
&g_n(\omega,t)={}\int_0^t e^{\omega s} \frac{\partial^n}{\partial x^n}q^{(1)}(0,s)\ud s, &n={}0,1,2,&&0<t<T,\\
&h_n(\omega,t)={}\int_0^t e^{\omega s} \frac{\partial^n}{\partial x^n}q^{(2)}(0,s)\ud s, &n={}0,1,2,&&0<t<T,\\
\end{align*}

The global relations become
\begin{subequations}\label{GR_2i}
\begin{align}
e^{\omega_1t}\hat{q}^{(1)}(k,t)={}&\hat{q}^{(1)}_0(k)+\sigma_1^3\left(g_2(\omega_1,t)+ikg_1(\omega_1,t)-k^2g_0(\omega_1,t)\right), &k_I\geq0,\\
e^{\omega_2t}\hat{q}^{(2)}(k,t)={}&\hat{q}^{(2)}_0(k)-\sigma_2^3\left(h_2(\omega_2,t)+ikh_1(\omega_2,t)-k^2h_0(\omega_2,t)\right), &k_I\leq0.
\end{align}
\end{subequations}
We wish to transform the global relations so that $g_n(\cdot,t)$ and $h_n(\cdot,t)$ for $n={}0,1,2$ depend on a common argument, $ik^3$ as in~\cite{Asvestas, Mantzavinos}. Noting $ik^3$ is invariant under the transformations $k\to\alpha k$ and $k\to\alpha^2k$ where $\alpha={}e^{2i\pi /3}$ and evaluating at $t={}T$ we have the following six global relations:

\begin{subequations}\label{tGR_2i}
\begin{align}
\begin{split}
e^{ik^3T}\hat{q}^{(1)}\left(\frac{\alpha^j k}{\sigma_1},T\right)={}&\hat{q}^{(1)}_0\left(\frac{\alpha^j k}{\sigma_1}\right)+\left(\sigma_1^3 g_2(ik^3,T)+i\alpha^j k\sigma_1^2g_1(ik^3,T)-\alpha^{2j}k^2\sigma_1g_0(ik^3,T)\right),\\
&\sigma_1\Im(\alpha^j k)\geq0,
\end{split}\\
\begin{split}
e^{ik^3T}\hat{q}^{(2)}\left(\frac{\alpha^j k}{\sigma_2},T\right)={}&\hat{q}^{(2)}_0\left(\frac{\alpha^j k}{\sigma_2}\right)-\left(\sigma_2^3 h_2(ik^3,T)+i\alpha^j k\sigma_2^2 h_1(ik^3,T)-\alpha^{2j}k^2\sigma_2 h_0(ik^3,T)\right), \\
&\sigma_2\Im(\alpha^jk)\leq0,
\end{split}
\end{align}
\end{subequations}
for $j={}0,1,2$.

Inverting the Fourier transform in~\eqref{GR_2i} we have the solution formulas
\begin{subequations}
\begin{equation}
\begin{split}
q^{(1)}(x,t)={}&\frac{1}{2\pi}\int_{-\infty}^\infty e^{ikx-\omega_1t}\hat{q}^{(1)}_0(k)\ud k\\
&+\frac{\sigma_1^3}{2\pi}\int_{-\infty}^\infty e^{ikx-\omega_1t} \left(g_2(\omega_1,t)+ikg_1(\omega_1,t)-k^2g_0(\omega_1,t)\right)\ud k,
\end{split}
\end{equation}
\begin{equation}
\begin{split}
q^{(2)}(x,t)={}&\frac{1}{2\pi}\int_{-\infty}^\infty e^{ikx-\omega_2t}\hat{q}^{(2)}_0(k)\ud k\\
&-\frac{\sigma_2^3}{2\pi}\int_{-\infty}^\infty e^{ikx-\omega_2t}\left(h_2(\omega_2,t)+ik h_1(\omega_2,t)-k^2h_0(\omega_2,t)\right)\ud k,
\end{split}
\end{equation}
\end{subequations}
for $t\in(0,T)$ and $x\in(-\infty,0)$ and $x\in(0,\infty)$ respectively.  Next, we transform the second integral in each of the previous equations so that $g_n(\cdot,t)$ and $h_n(\cdot,t)$ depend on $ik^3$ for $n={}0,1,2$.  

\begin{subequations}\label{solns_2i}
\begin{equation}\label{soln1_2i}
\begin{split}
q^{(1)}(x,t)={}&\frac{1}{2\pi}\int_{-\infty}^\infty e^{ikx-\omega_1t}\hat{q}^{(1)}_0(k)\ud k\\
&+\frac{1}{2\pi}\int_{-\infty}^\infty e^{i\frac{k}{\sigma_1}x-ik^3t} \left(\sigma_1^2g_2(ik^3,t)+ik\sigma_1g_1(ik^3,t)-k^2g_0(ik^3,t)\right)\ud k,
\end{split}
\end{equation}
\begin{equation}\label{soln2_2i} 
\begin{split}
q^{(2)}(x,t)={}&\frac{1}{2\pi}\int_{-\infty}^\infty e^{ikx-\omega_2t}\hat{q}^{(2)}_0(k)\ud k\\
&-\frac{1}{2\pi}\int_{-\infty}^\infty e^{i\frac{k}{\sigma_2}x-ik^3t}\left(\sigma_2^2h_2(ik^3,t)+ik \sigma_2 h_1(ik^3,t)-k^2h_0(ik^3,t)\right)\ud k,
\end{split}
\end{equation}
\end{subequations}

Let $D={}\{k\in\CC:\Re(ik^3)<0\}={}D^{(1)}\cup D^{(3)}\cup D^{(5)}$ as in Figure~\ref{fig:LKdV_D}. The parenthetical numbers in the superscript of $D$ represent an enumeration of the sectors of the complex plane, in contrast to the parenthetical numbers in the superscript of $q$ (and $\Gamma$, below), which represent the two half-line domains $(-\infty,0)$ and $(0,\infty)$.  Let $$D_R={}\{k\in \CC: k\in D \mbox{ and } |k|>R\}={}D_R^{(1)} \cup D_R^{(3)} \cup D_R^{(5)},$$ where $R>0$ is a positive constant as shown in Figure~\ref{fig:LKdV_DR}.  Let $\Gamma^{(j)}$ be the contour $\partial \{ k \in D_R : (-1)^j \sigma_j k_I > 0 \}$, oriented so that $D_R^{(1)}$ and $D_R^{(3)}$ lie to the left, and $D_R^{(5)}$ lies to the right of any $\Gamma^{(j)}$ to which they are adjacent. Note that whether $\Gamma^{(j)}$ is the boundary of $D_R^{(1)} \cup D_R^{(3)}$ or $D_R^{(5)}$ depends not only on $j$ but also upon the sign of $\sigma_j$. The integrand of the second integral in~\eqref{soln1_2i} is analytic and decays as $k\to\infty$ from within the set bounded between $\mathbb{R}$ and $\Gamma^{(1)}$, and the integrand of the second integral in~\eqref{soln2_2i} is analytic and decays as $k\to\infty$ from within the set bounded between $\mathbb{R}$ and $\Gamma^{(2)}$. Hence, by Jordan's Lemma and Cauchy's Theorem, the contours of integration can be deformed from $\mathbb{R}$ to $\Gamma^{(j)}$.

 \begin{figure} 
   \centering
   \def\svgwidth{3in}
   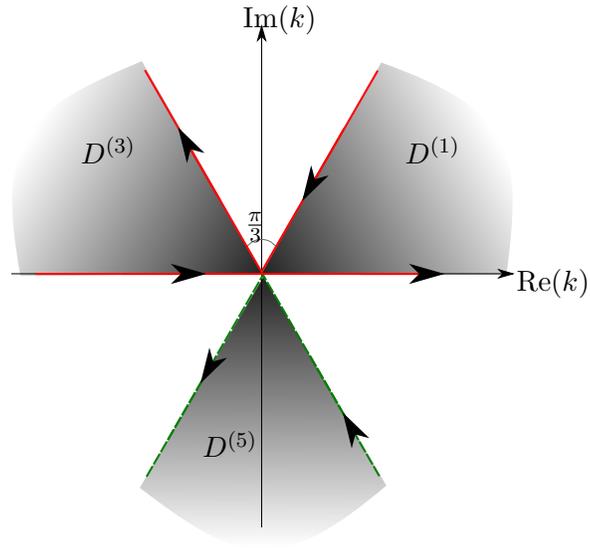
      \caption{The evenly distributed regions $D^{(1)}$, $D^{(3)}$, $D^{(5)}$ where $\Re(ik^3)<0$.
   \label{fig:LKdV_D}}
\end{figure} 
 \begin{figure} 
   \centering
   \def\svgwidth{3in}
   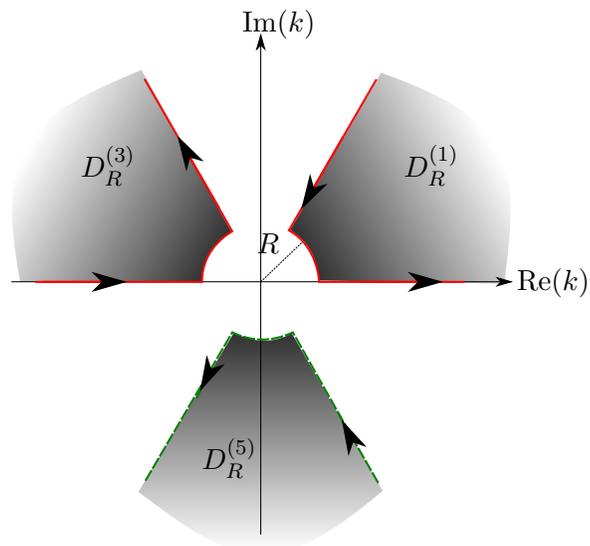
      \caption{The regions $D_R^{(1)}$, $D_R^{(3)}$, $D_R^{(5)}$ where $\Re(ik^3)<0$ and $|k|>R$.
   \label{fig:LKdV_DR}}
\end{figure} 

\begin{subequations}\label{dsolns_2i_t}
\begin{equation}\label{dsoln1_2i}
\begin{split}
q^{(1)}(x,t)={}&\frac{1}{2\pi}\int_{-\infty}^\infty e^{ikx-\omega_1t}\hat{q}^{(1)}_0(k)\ud k\\
&+\frac{1}{2\pi}\int_{\Gamma^{(1)}} e^{i\frac{k}{\sigma_1}x-ik^3t} \left(\sigma_1^2g_2(ik^3,t)+ik\sigma_1g_1(ik^3,t)-k^2g_0(ik^3,t)\right)\ud k,
\end{split}
\end{equation}
\begin{equation}\label{dsoln2_2i} 
\begin{split}
q^{(2)}(x,t)={}&\frac{1}{2\pi}\int_{-\infty}^\infty e^{ikx-\omega_2t}\hat{q}^{(2)}_0(k)\ud k\\
&-\frac{1}{2\pi}\int_{\Gamma^{(2)}} e^{i\frac{k}{\sigma_2}x-ik^3t}\left(\sigma_2^2h_2(ik^3,t)+ik \sigma_2 h_1(ik^3,t)-k^2h_0(ik^3,t)\right)\ud k.
\end{split}
\end{equation}
\end{subequations}
We replace $t$ by $T$ in the arguments of $g_j$ and $h_j$ by noting that this is equivalent to replacing the integral $\int_0^t e^{ik^3} \frac{\partial^n}{\partial x^n} q^{(j)}(0,s)\ud s$ with $\int_0^T e^{ik^3} \frac{\partial^n}{\partial x^n} q^{(j)}(0,s)\ud s-\int_t^T e^{ik^3} \frac{\partial^n}{\partial x^n} q^{(j)}(0,s)\ud s$.  Using analyticity properties of the integrand and Jordan's Lemma, the contribution from the second integral is zero and thus, 

\begin{subequations}\label{dsolns_2i_T}
\begin{equation}
\begin{split}
q^{(1)}(x,t)={}&\frac{1}{2\pi}\int_{-\infty}^\infty e^{ikx-\omega_1t}\hat{q}^{(1)}_0(k)\ud k\\
&+\frac{1}{2\pi}\int_{\Gamma^{(1)}} e^{i\frac{k}{\sigma_1}x-ik^3t} \left(\sigma_1^2g_2(ik^3,T)+ik\sigma_1g_1(ik^3,T)-k^2g_0(ik^3,T)\right)\ud k,
\end{split}
\end{equation}
\begin{equation}
\begin{split}
q^{(2)}(x,t)={}&\frac{1}{2\pi}\int_{-\infty}^\infty e^{ikx-\omega_2t}\hat{q}^{(2)}_0(k)\ud k\\
&-\frac{1}{2\pi}\int_{\Gamma^{(2)}} e^{i\frac{k}{\sigma_2}x-ik^3t}\left(\sigma_2^2h_2(ik^3,T)+ik \sigma_2 h_1(ik^3,T)-k^2h_0(ik^3,T)\right)\ud k.
\end{split}
\end{equation}
\end{subequations}
While Equation~\eqref{dsolns_2i_t} makes the functional dependence of the solution more complicated than in Equation~\eqref{dsolns_2i_T}, it is useful for doing long time asymptotics, \emph{i.e.} taking the limit as $t\to\infty$.  Equation~\eqref{dsolns_2i_T} is useful for checking that the expression satisfies the equation.  While the integrands of these expressions are different, the integrals are equal and thus one may switch between them whenever convenient.  

In Section~\ref{sec:results}, we show how it is possible to obtain expressions for all six spectral functions $g_j$, $h_j$ in the relevant domains by solving a linear system.   Indeed, for any $r\in\{1,3,5\}$, if $k\in \overline{D}_R^{(r)}$ (the closure of $D_R^{(r)}$), then a certain number, say $m$, of the global relation equations~\eqref{GR_2i} are valid for $k$. We must supplement these equations with $6-m$ interface conditions to obtain a solvable system. Given the coefficients of $g_j$, $h_j$ in~\eqref{GR_2i}, it is clear that the determinant of the linear system must be a polynomial in $k$. The criteria of Propositions~\ref{prop:sigma_pn}--\ref{prop:sigma_np} identify the cases in which this determinant is not identically $0$, that is the system is full rank. For such a full rank system, it is always possible to choose $R>0$ sufficiently large that $\overline{D}_R$ contains no zeros of the determinant, which is essential in the proof of Proposition~\ref{prop:soln}.  We denote this linear system by

\begin{equation}\label{linsystem}
\mathcal{A}X=Y+e^{ik^3T}\mathcal{Y},
\end{equation}
where 
\begin{equation}\label{LKdV_X}
X={}\left(g_0(ik^3,T),g_1(ik^3,T),g_2(ik^3,T),h_0(ik^3,T),h_1(ik^3,T),h_2(ik^3,T)\right)^\top.
\end{equation}
The right-hand side of~\eqref{linsystem} is expressed as the sum of $Y$, which includes expressions that are known explicitly (\emph{i.e.}, $q^{(j)}_0(\cdot)$, $j={}1,2$ and non-homogenous terms from the interface conditions) and $\mathcal{Y}$ which includes unknown expressions (\emph{i.e.}, $\hat{q}^{(j)}(\cdot,T)$, $j={}1,2$).

\section{Results}\label{sec:results}
In each of the following propositions, we assume that the interface conditions are not such that the problem reduces to a pair of BVP. It is a matter of trivial linear algebra to determine whether any particular problem has this property, and its well-posedness and solution are then known~\cite{FokasBook, WangFokas}.

In the case $\sigma_1>0$ and $\sigma_2<0$, the phase velocity for $x<0$ is positive and the phase velocity for $x>0$ is negative.  Thus, information from the initial conditions propagates toward the interface as in Figure~\ref{fig:sink_sink}.  In this case we expect the minimal number of interface conditions to be necessary for a well-posed problem.
\begin{figure}[htbp]
\begin{center}
\def\svgwidth{4in}
   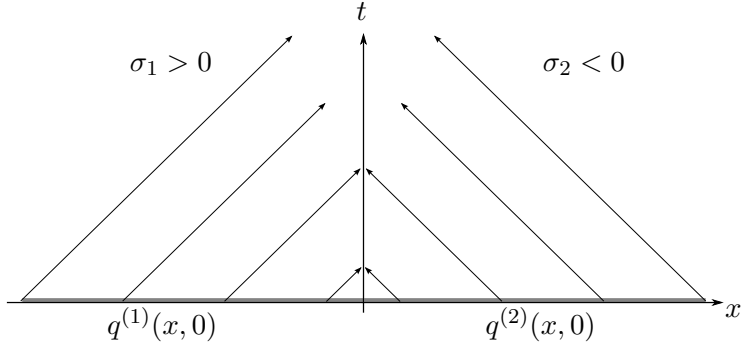 
   \caption{Information from the initial conditions $q^{(1)}(x,0)$ and $q^{(2)}(x,0)$ propagates toward the interface.    \label{fig:sink_sink}}
  \end{center}
\end{figure}

\begin{prop}\label{prop:sigma_pn}
Assume $\sigma_1>0$ and $\sigma_2<0$.  Equation~\eqref{linsystem} is solvable for $X$ if and only if two interface conditions are given.  These conditions must be of the form
\begin{subequations}\label{ics_pn}
\begin{align}
\beta_{11}q^{(1)}(0,t)+\beta_{12} q^{(1)}_{x}(0,t)+\beta_{13} q^{(1)}_{xx}(0,t)+\beta_{14}q^{(2)}(0,t)+\beta_{15}q^{(2)}_x(0,t)+\beta_{16}q^{(2)}_{xx}(0,t)&={}f_1(t),\\
\beta_{21}q^{(1)}(0,t)+\beta_{22}q^{(1)}_x(0,t)+\beta_{23} q^{(1)}_{xx}(0,t)+\beta_{24}q^{(2)}(0,t)+\beta_{25}q^{(2)}_x(0,t)+\beta_{26}q^{(2)}_{xx}(0,t)&={}f_2(t).
\end{align}
\end{subequations}

The solution to~\eqref{linsystem} is full rank, that is, solvable for $X$, whenever at least one of the following holds
\begin{enumerate}
\item $\beta_{14}\beta_{21}\neq\beta_{11}\beta_{24}$,
\item $\sigma_1(\beta_{15}\beta_{21}-\beta_{11}\beta_{25})\neq \sigma_2(\beta_{12}\beta_{24}-\beta_{14}\beta_{22}),$
\item $\sigma_1^2(\beta_{16}\beta_{21}-\beta_{11}\beta_{26})+\sigma_1\sigma_2(\beta_{15}\beta_{22}-\beta_{12}\beta_{25})+\sigma_2^2(\beta_{14}\beta_{24}-\beta_{13}\beta_{24})\neq0,$
\item $\sigma_1(\beta_{16}\beta_{22}-\beta_{12}\beta_{26})\neq\sigma_2(\beta_{13}\beta_{25}-\beta_{15}\beta_{23})$,
\item $\beta_{16}\beta_{23}\neq\beta_{13}\beta_{26}$.
\end{enumerate}
\end{prop}

\begin{proof}[Proof of Proposition~\ref{prop:sigma_pn}]
In the case $\sigma_1>0$ and $\sigma_2<0$, the second integrals of both~\eqref{soln1_2i} and~\eqref{soln2_2i} can be deformed from $\int_{-\infty}^\infty \cdot \ud k$ to $-\int_{\partial D_R^{(5)}}\cdot \ud k$. We rewrite the global relations~\eqref{GR_2i} as
\begin{subequations}\label{pn_GR}
\begin{equation}\label{pn_GR_1_rp2}
	 e^{ik^3t}\hat{q}^{(1)}\left(\frac{\alpha k}{\sigma_1},T\right)-\hat{q}^{(1)}_0\left(\frac{\alpha k}{\sigma_1}\right) ={} \sigma_1^3 g_2(ik^3,T) + i\alpha k\sigma_1^2g_1(ik^3,T) - (\alpha k)^2\sigma_1g_0(ik^3,T),
\end{equation}
\begin{equation}\label{pn_GR_2_rp2}
	e^{ik^3t}\hat{q}^{(2)}\left(\frac{\alpha k}{\sigma_2},T\right)-\hat{q}^{(2)}_0\left(\frac{\alpha k}{\sigma_2}\right) ={} -\sigma_2^3 h_2(ik^3,T) - i\alpha k\sigma_2^2h_1(ik^3,T) + (\alpha k)^2\sigma_2h_0(ik^3,T),
\end{equation}
\begin{equation}\label{pn_GR_1_r}
	 e^{ik^3t}\hat{q}^{(1)}\left(\frac{\alpha^2 k}{\sigma_1},T\right)-\hat{q}^{(1)}_0\left(\frac{\alpha^2 k}{\sigma_1}\right) ={} \sigma_1^3 g_2(ik^3,T) + i\alpha^2 k\sigma_1^2g_1(ik^3,T) - (\alpha^2 k)^2\sigma_1g_0(ik^3,T),
\end{equation}
\begin{equation}\label{pn_GR_2_r}
	e^{ik^3t}\hat{q}^{(2)}\left(\frac{\alpha^{2}k}{\sigma_2},T\right)-\hat{q}^{(2)}_0\left(\frac{\alpha^{2}k}{\sigma_2}\right)  ={}- \sigma_2^3 h_2(ik^3,T) - ik\alpha^{2}\sigma_2^2h_1(ik^3,T) + (\alpha^{2}k)^2\sigma_2h_0(ik^3,T),
\end{equation}
\end{subequations}
which are all valid for $k\in\overline{D}^{(5)}$.  Evaluating~\eqref{ics_pn} for $t={}s$, multiplying by $e^{ik^3s}$, and integrating from $0$ to $t$ one obtains
\begin{equation*}\label{eqn:t_ics_pn}
	h_j(ik^3,T) + \sum_{\ell={}0}^2 \beta_{j+1 ,\ell+1} g_\ell(ik^3,T) ={} \tilde{f}_{j+1}(ik^3,T),\quad j\in\{0,1,2\},
\end{equation*}
where
\begin{equation*}
\tilde{f}_j (\omega,T)={}\int_0^T e^{\omega s} f_j(s)\ud s,~~j\in\{0,1,2\},
\end{equation*}
which is valid for $k \in \overline{D}^{(r)}$ (the closure of $D^{(r)}$). 

In order to solve the full $6\times 6$ system it is clear we must impose two ``interface conditions," since the global relations~\eqref{pn_GR} provide exactly four of the necessary six equations.  If there is one boundary condition relating $q^{(1)}$ and $q^{(2)}$ and their spatial derivatives then one can solve the problem on the left (right) and use the solution and remaining interface conditions to solve the problem on the right (left).  The half-line problem is well posed~\cite{FokasBook, WangFokas} and its solution will not be considered here.  Hence, ``interface conditions" of the type~\eqref{ics_pn} are all we need to consider.  

The above argument only fails if $\det(\mathcal{A})\equiv0$ since all singularities are outside $D_R^{(5)}$.  Examining $\det(\mathcal{A})={}0$ one obtains a polynomial in $k$.  Since we need this to hold for all $k$, we consider the coefficients of each power of $k$.  Requiring at least one coefficient to be nonzero gives the conditions stated in~Proposition~\eqref{prop:sigma_pn}.  
\end{proof}

In the case $\sigma_1>0$ and $\sigma_2>0$, the phase velocity for $x<0$ and $x>0$ is positive.  Thus, information from $q_0^{(1)}(x)$ propagates toward the interface but information from $q_0^{(2)}(x)$ propagates away from the interface as in Figure~\ref{fig:sink_source}.  Hence, we expect that more interface conditions are necessary for a well-posed problem than in the case when $\sigma_1>0$ and $\sigma_2<0$ as in Proposition~\ref{prop:sigma_pn}.  Notice that the case of $\sigma_1<0$ and $\sigma_2<0$ could be considered in this case by letting $x\to-x$.  Hence, we consider only the case where $\sigma_1>0$ and $\sigma_2>0$.

\begin{figure}[htbp]
\begin{center}
\def\svgwidth{4in}
   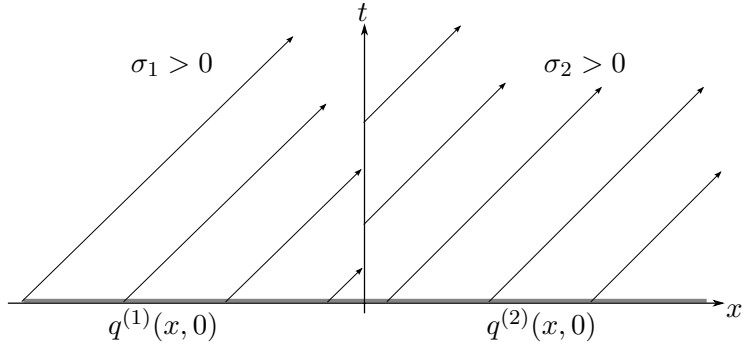 
   \caption{Information from the initial condition $q^{(1)}(x,0)$ propagates toward the interface while information from $q^{(2)}(x,0)$ propagates away from the interface.    \label{fig:sink_source}}
  \end{center}
\end{figure}

\begin{prop}\label{prop:sigma_pp}
Assume $\sigma_1>0$ and $\sigma_2>0$.  The square matrix $\mathcal{A}$ in~\eqref{linsystem} is solvable for $X$ if and only if three interface conditions of the following form are given.  
\begin{subequations}\label{ics_pp}
\begin{align}
\beta_{11}q^{(1)}(0,t)+\beta_{12}q^{(1)}_x(0,t)+\beta_{13}q^{(1)}_{xx}(0,t)+q^{(2)}(0,t)&={}f_1(t),\\
\beta_{21}q^{(1)}(0,t)+\beta_{22}q^{(1)}_x(0,t)+\beta_{23}q^{(1)}_{xx}(0,t)+q^{(2)}_x(0,t)&={}f_2(t),\\
\beta_{31}q^{(1)}(0,t)+\beta_{32}q^{(1)}_x(0,t)+\beta_{33}q^{(1)}_{xx}(0,t)+q^{(2)}_{xx}(0,t)&={}f_3(t).
\end{align}
\end{subequations}

The solution to~\eqref{linsystem} is solvable for $X$ whenever one or more of the following is satisfied:
\begin{enumerate}
\item $\beta_{31}\neq0$,
\item $\sigma_1\beta_{21}+\sigma_2\beta_{32}\neq0$,
\item $\sigma_1^2\beta_{11}+\sigma_1\sigma_2\beta_{22}+\sigma_2^2\beta_{33}\neq0$,
\item $\sigma_1\beta_{12}+\sigma_2\beta_{23}\neq0$,
\item $\beta_{13}\neq0$.
\end{enumerate}

\end{prop}

\vspace{.4in}
{\bf Remark.}
It may be possible to rewrite the interface conditions so that one is a boundary condition for $q^{(2)}$ and still have an interface problem. However, a single boundary condition for $q^{(1)}$ or a pair of boundary conditions for $q^{(2)}$ implies that the problem separates into a pair of BVPs.
\vspace{.4in}

\begin{proof}[Proof of Proposition~\ref{prop:sigma_pp}]
In the case $\sigma_1>0$ and $\sigma_2>0$, the second integrals of~\eqref{soln1_2i} can be deformed from $\int_{-\infty}^\infty \cdot \ud k$ to $-\int_{\partial D_R^{(5)}}\cdot \ud k$.  The second integral of~\eqref{soln2_2i} can be deformed from $\int_{-\infty}^\infty \cdot \ud k$ to $\int_{\partial D_R^{(1)}}\cdot \ud k+\int_{\partial D_R^{(3)}}\cdot \ud k$.  We rewrite the global relations for each $r\in\{1,3,5\}$ as
\begin{subequations}\label{pp_GR_r}
\begin{equation}\label{pp_GR_1_r}
	 e^{ik^3t}\hat{q}^{(1)}\left(\frac{\alpha^{r}k}{\sigma_1},T\right)-\hat{q}^{(1)}_0\left(\frac{\alpha^{r}k}{\sigma_1}\right) ={} \sigma_1^3 g_2(ik^3,T) + ik\alpha^{r}\sigma_1^2g_1(ik^3,T) - (k\alpha^{r})^2\sigma_1g_0(ik^3,T),
\end{equation}
\begin{equation}\label{pp_GR_1_rp2}
\begin{split}
	e^{ik^3t}\hat{q}^{(1)}\left(\frac{\alpha^{r+2}k}{\sigma_1},T\right)-{}&\hat{q}^{(1)}_0\left(\frac{\alpha^{r+2}k}{\sigma_1}\right)  ={}\\
	& \sigma_1^3 g_2(ik^3,T) + ik\alpha^{r+2}\sigma_1^2g_1(ik^3,T) - (k\alpha^{r+2})^2\sigma_1g_0(ik^3,T),
	\end{split}
\end{equation}
\begin{equation}\label{pp_GR_2_rp1}
\begin{split}
	e^{ik^3t}\hat{q}^{(2)}\left(\frac{\alpha^{r+1}k}{\sigma_2},T\right)-{}&\hat{q}^{(2)}_0\left(\frac{\alpha^{r+1}k}{\sigma_2}\right)  ={}\\
	&- \sigma_2^3 h_2(ik^3,T) - ik\alpha^{r+1}\sigma_2^2h_1(ik^3,T) + (k\alpha^{r+1})^2\sigma_2h_0(ik^3,T),
	\end{split}
\end{equation}
\end{subequations}
which are all valid for $k\in\overline{D}^{(r)}$.  Evaluating~\eqref{ics_pp} for $t=s$, multiplying by $e^{ik^3s}$, and integrating from $0$ to $t$ one obtains
\begin{equation*}\label{eqn:t_ics_pp}
	h_j(ik^3,T) + \sum_{\ell=0}^2 \beta_{j+1 ,\ell+1} g_\ell(ik^3,T) =\tilde{f}_{j+1}(ik^3,T),\quad j\in\{0,1,2\},
\end{equation*}
where
\begin{equation*}
\tilde{f}_j (\omega,T)={}\int_0^T e^{\omega s} f_j(s)\ud s,~~j\in\{0,1,2\},
\end{equation*}
which is valid for $k \in \overline{D}^{(r)}$. 

In order to solve the full $6\times 6$ system it is clear we must impose three ``interface conditions", since~\eqref{pp_GR_r} provides just three equations.  We must now examine the cases where one or more of these conditions decouples into a boundary condition on either $q^{(1)}$ or $q^{(2)}$.  If there is one boundary condition relating $q^{(1)}$ and its spatial derivatives, then one can solve the problem on the left and use the solution and remaining interface conditions to solve the problem on the right.  Solving the half-line problem is well posed~\cite{FokasBook, WangFokas} and is not considered here.  Hence, ``interface conditions" of the type~\eqref{ics_pp} are all we need to consider.  

Examining $\det(\mathcal{A})={}0$ in this case, one obtains a polynomial in $k$.  Since we need this to hold for all $k$, we consider the coefficients of each power of $k$.  Since we want conditions on $\det(\mathcal{A})\neq0$ we need at least one of the coefficients to be nonzero.  This gives the conditions stated in Proposition~\eqref{prop:sigma_pp}.

\end{proof}

In the case $\sigma_1<0$ and $\sigma_2>0$, the phase velocity for $x<0$ is negative and the phase velocity for $x>0$ is positive.  Thus, information from the initial conditions propagates away from the interface as in Figure~\ref{fig:source_source}.  Hence, we expect that more interface conditions are necessary for a well-posed problem than in the previous cases.

\begin{figure}[htbp]
\begin{center}
\def\svgwidth{4in}
   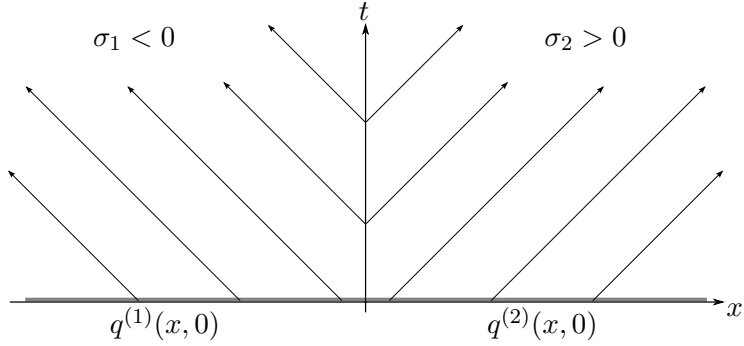 
   \caption{Information from the initial conditions $q^{(1)}(x,0)$ and $q^{(2)}(x,0)$ propagates away from the interface.    \label{fig:source_source}}
  \end{center}
\end{figure}

\begin{prop}\label{prop:sigma_np}
Assume $\sigma_1<0$ and $\sigma_2>0$.  Equation~\eqref{linsystem} is full rank if and only if four interface conditions are given.  These conditions must be of the form
\begin{subequations}\label{ics_np}
\begin{align}
q^{(1)}(0,t)+\beta_{14}q^{(2)}(0,t)+\beta_{15}q^{(2)}_x(0,t)+\beta_{16}q^{(2)}_{xx}(0,t)&={}f_1(t),\label{ics_np_ic1}\\
q^{(1)}_x(0,t)+\beta_{24}q^{(2)}(0,t)+\beta_{25}q^{(2)}_x(0,t)+\beta_{26}q^{(2)}_{xx}(0,t)&={}f_2(t),\label{ics_np_ic2}\\
q^{(1)}_{xx}(0,t)+\beta_{34}q^{(2)}(0,t)+\beta_{35}q^{(2)}_x(0,t)+\beta_{36}q^{(2)}_{xx}(0,t)&={}f_3(t),\label{ics_np_ic3}\\
\beta_{44}q^{(2)}(0,t)+\beta_{45}q^{(2)}_x(0,t)+\beta_{46}q^{(2)}_{xx}(0,t)&={}f_4(t).\label{ics_np_bc}
\end{align}
\end{subequations}

The solution to~\eqref{linsystem} is full rank whenever one or more of the following is satisfied:
\begin{enumerate}
\item $\beta_{35}\beta_{44}\neq\beta_{34}\beta_{45}$,
\item $\sigma_1(\beta_{34}\beta_{46}-\beta_{36}\beta_{44})\neq\sigma_2(\beta_{24}\beta_{45}-\beta_{25}\beta_{44})$,
\item $\sigma_1^2(\beta_{35}\beta_{46}-\beta_{36}\beta_{45})+\sigma_1\sigma_2(\beta_{26}\beta_{44}-\beta_{24}\beta_{46})+\sigma_2^2(\beta_{14}\beta_{45}-\beta_{15}\beta_{44})\neq0$,
\item $\sigma_1(\beta_{26}\beta_{45}-\beta_{25}\beta_{46})\neq\sigma_2(\beta_{16}\beta_{44}-\beta_{14}\beta_{46})$,
\item $\beta_{16}\beta_{45}\neq\beta_{15}\beta_{46}$.
\end{enumerate}

\end{prop}

\vspace{.4in}
{\bf Remark.}
As four interface conditions are required, it must be possible to write (at least) two as boundary conditions. If there are two boundary conditions for either $q^{(1)}$ or $q^{(2)}$, then the problem separates into a pair of BVP, so we only consider the case where there is precisely one boundary condition for each of $q^{(1)}$ and $q^{(2)}$. However, for the purposes of stating the result, it is more convenient to write the conditions in the form~\eqref{ics_np}.
\vspace{.4in}

\begin{proof}[Proof of Proposition~\ref{prop:sigma_np}]
In the case $\sigma_1<0$ and $\sigma_2>0$, the second integrals of both~\eqref{soln1_2i} and~\eqref{soln2_2i} can be deformed from $\int_{-\infty}^\infty \cdot \ud k$ to $\int_{\partial D_R^{(1)}}\cdot \ud k+\int_{\partial D_R^{(3)}}\cdot \ud k$.  The appropriate global relations can be rewritten for $r\in\{1,3\}$ as
\begin{subequations}\label{np_GR_r}
\begin{equation}\label{np_GR_1_rp1}
\begin{split}
	e^{ik^3t}\hat{q}^{(1)}\left(\frac{\alpha^{r+1}k}{\sigma_1},T\right) -{}&\hat{q}^{(1)}_0\left(\frac{\alpha^{r+1}k}{\sigma_1}\right)  ={} \\
	&\sigma_1^3 g_2(ik^3,T) + ik\alpha^{r+1}\sigma_1^2g_1(ik^3,T) - (k\alpha^{r+1})^2\sigma_1g_0(ik^3,T),
	\end{split}
\end{equation}
\begin{equation}\label{np_GR_2_rp1}
\begin{split}
	e^{ik^3t}\hat{q}^{(2)}\left(\frac{\alpha^{r+1}k}{\sigma_2},T\right)-{}&\hat{q}^{(2)}_0\left(\frac{\alpha^{r+1}k}{\sigma_2}\right)  ={}\\
	&- \sigma_2^3 h_2(ik^3,T) - ik\alpha^{r+1}\sigma_2^2h_1(ik^3,T)+(k\alpha^{r+1})^2\sigma_2h_0(ik^3,T),
	\end{split}
\end{equation}
\end{subequations}
which are all valid for $k\in\overline{D}^{(r)}$.  Evaluating~\eqref{ics_np} for $t={}s$, multiplying by $e^{ik^3s}$, and integrating from $0$ to $t$ one obtains
\begin{align*}
g_0(ik^3,T)+\beta_{15}h_1(ik^3,T)+\beta_{16}h_2(ik^3,T)&={}\tilde{f}_1(ik^3,T),\\
g_1(ik^3,T)+\beta_{25}h_1(ik^3,T)+\beta_{26}h_2(ik^3,T)&={}\tilde{f}_2(ik^3,T),\\
g_2(ik^3,T)+\beta_{35}h_1(ik^3,T)+\beta_{36}h_2(ik^3,T)&={}\tilde{f}_3(ik^3,T),\\
h_0(ik^3,T)+\beta_{45}h_1(ik^3,T)+\beta_{46}h_2(ik^3,T)&={}\tilde{f}_4(ik^3,T),
\end{align*}
where
\begin{equation*}
\tilde{f}_j (\omega,T)={}\int_0^T e^{\omega s} f_j(s)\ud s,~~j\in\{1,2,3,4\},
\end{equation*}
which is valid for $k \in \overline{D}^{(r)}$. 

In order to solve the full $6\times 6$ system~\eqref{linsystem} we must impose four ``interface conditions,"  since~\eqref{np_GR_r} gives only two equations. We need to examine the cases where one or more of these conditions decouples into a boundary condition on either $q^{(1)}$ or $q^{(2)}$.  Using elementary linear algebra it is clear that at least one of these conditions must be a boundary condition.  If there are two boundary conditions relating $q^{(1)}$ ($q^{(2)}$) and its spatial derivatives then one can solve the problem on the left (right) and use the solution and remaining interface conditions to solve the problem on the right (left).  Solving the half-line problem is well posed~\cite{FokasBook, WangFokas} and will not be considered here. 

There must be precisely one boundary condition of the form~\eqref{ics_np_bc}.  It is then elementary linear algebra to see that the remaining three conditions can be written as interface conditions in the form~\eqref{ics_np_ic1}-\eqref{ics_np_ic3}.

Examining $\det(\mathcal{A})={}0$ one obtains a polynomial in $k$.  We need this condition to hold for all $k$ and we consider the coefficients of each power of $k$.  Since we want conditions on $\det(\mathcal{A})\neq0$ we need at least one of the coefficients to be nonzero.  This gives the conditions stated in Proposition~\eqref{prop:sigma_np}.
\end{proof}

\begin{prop}\label{prop:soln}
Assume $\mathcal{A}$ in~\eqref{linsystem} is full rank.  A solution to~\eqref{LKdV_2i} is given by~\eqref{dsolns_2i_T} where $g_j(ik^3,T)$ and $h_j(ik^3,T)$ for $j={}0,1,2$ are the solution to the linear system $\mathcal{A}X={}Y$ where $\mathcal{A}$, $X$, and $Y$ are given in~\eqref{LKdV_X} and the surrounding paragraph.
\end{prop}

\begin{proof}[Proof of Proposition~\ref{prop:soln}]
Consider $\mathcal{A}_{j}$, which is the matrix $\mathcal{A}$ with the $j^\textrm{th}$ column replaced by $e^{ik^3T}\mathcal{Y}$.  We  solve $\mathcal{A}X=e^{ik^3T}\mathcal{Y}$ using Cramer's Rule~\cite{Cramer}.  If we show that the contribution to the solution from $e^{ik^3T}\mathcal{Y}$ is zero, then we have proved the proposition.  The terms we are concerned with from~\eqref{dsolns_2i_T} are
$$\frac{1}{2\pi}\int_{\Gamma^{(1)}} e^{i\frac{k}{\sigma_1}x-ik^3t} \left(\sigma_1^2g_2(ik^3,T)+ik\sigma_1g_1(ik^3,T)-k^2g_0(ik^3,T)\right)\ud k.$$
and
$$\frac{1}{2\pi}\int_{\Gamma^{(2)}} e^{i\frac{k}{\sigma_2}x-ik^3t} \left(\sigma_2^2h_2(ik^3,T)+ik\sigma_2h_1(ik^3,T)-k^2h_0(ik^3,T)\right)\ud k.$$
Using Cramer's Rule and our factorization these become
\begin{subequations}\label{crule}
\begin{equation}\label{crule1}
\frac{1}{2\pi}\int_{\Gamma^{(1)}} e^{i\frac{k}{\sigma_1}x+ik^3(T-t)} \left(\sigma_1^2 \frac{\det(\mathcal{A}_{3})}{\det(\mathcal{A})}+ik\sigma_1\frac{\det(\mathcal{A}_{2})}{\det(\mathcal{A})}-k^2\frac{\det(\mathcal{A}_{1})}{\det(\mathcal{A})}\right)\ud k.
\end{equation}
\begin{equation}\label{crule2}
\frac{1}{2\pi}\int_{\Gamma^{(2)}} e^{i\frac{k}{\sigma_2}x+ik^3(T-t)} \left(\sigma_2^2 \frac{\det(\mathcal{A}_{6})}{\det(\mathcal{A})}+ik\sigma_2\frac{\det(\mathcal{A}_{5})}{\det(\mathcal{A})}-k^2\frac{\det(\mathcal{A}_{4})}{\det(\mathcal{A})}\right)\ud k.
\end{equation}
\end{subequations}

We would like to show these integrand terms are analytic and decay for large $k$ inside the domains around which they are integrated.  Note that $\det(\mathcal{A})\neq0$ since~\eqref{linsystem} is full rank.

\begin{description}
\item[Case 1. $\sigma_1<0, \sigma_2>0$:]

For $\sigma_1<0$ and $\sigma_2>0$, $\Gamma^{(1)}=\Gamma^{(2)}=\partial D_R^{(1)}\cup\partial D_R^{(3)}$.  Using the form of $\mathcal{Y}$ in this case each term of the integrand in~\eqref{crule1} is of the form
\begin{align*}
{}&\frac{1}{2\pi}\int_{\Gamma^{(1)}} e^{i\frac{k x}{\sigma_1}+ik^3(T-t)}\left(c_1(k) \hat{q}^{(1)}\left(\frac{\alpha^{r+1}k}{\sigma_1},T\right)+c_2(k) \hat{q}^{(2)}\left(\frac{\alpha^{r+1}k}{\sigma_2},T\right) \right)\ud k\\
={}&\frac{1}{2\pi} \int_{\Gamma^{(1)}} e^{ik^3(T-t)+\frac{ikx}{\sigma_1}} c_1(k)\left( \int_{-\infty}^0 q^{(1)}(y,T)e^{\frac{-ik \alpha^{r+1} y}{\sigma_1}}\right)\ud k\ud y\\
&+\frac{1}{2\pi}\int_{\Gamma^{(1)}} e^{ik^3(T-t)+\frac{ikx}{\sigma_1}} c_2(k) \left(\int_{0}^\infty q^{(2)}(y,T)e^{\frac{-ik\alpha^{r+1} y}{\sigma_2}}\right)\ud k\ud y,
\end{align*}
where $r\in\{1,3\}$ depending on which region one is integrating around ($D_R^{(1)}$ or $D_R^{(3)}$) and $c_1(k)$ and $c_2(k)$ involve the constants $\beta_{j,\ell}$ which are $\mathcal{O}(1)$ as $k\to\infty$ from within $\Gamma^{(1)}$ and analytic for all $k\in D_R^{(r)}$.  For $k\to\infty$ with $k\in \overline{D}^{(r)}$ the expression inside the parenthesis decays by the Riemann-Lebesgue Lemma.  Thus, by Jordan's Lemma, these integrals along a closed, bounded curve in the complex $k$ plane vanish for $x<0$.  In particular we consider the closed curves  $\mathcal{L}^{(1)}={}\mathcal{L}_{D^{(1)}}\cup\mathcal{L}_{C^{(1)}}$ and $\mathcal{L}^{(3)}=\mathcal{L}_{D^{(3)}}\cup\mathcal{L}_{C^{(3)}}$ where $\mathcal{L}_{D^{(j)}}={}\partial D_R^{(j)} \cap \{k: |k|<C\}$ and $\mathcal{L}_{C^{(j)}}={}\{k\in D_R^{(j)}: |k|={}C\}$, see Figure~\ref{fig:LKdV_DR_close}.

\begin{figure}[tb]
   \centering
\def\svgwidth{3.5in}
   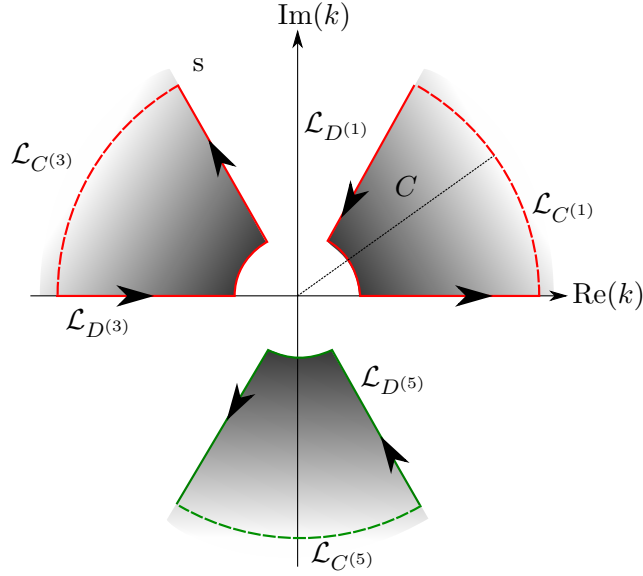 
   \caption[The contours $\mathcal{L}^{(1)}$, $\mathcal{L}^{(3)}$, and $\mathcal{L}^{(5)}$ for the linear KdV equation.]{The contours $\mathcal{L}_{D^{(1)}}$ and $\mathcal{L}_{D^{(3)}}$ are shown as red solid lines and the contours $\mathcal{L}_{C^{(1)}}$ and $\mathcal{L}_{C^{(3)}}$ are shown as red dashed lines.  The contour $\mathcal{L}_{D^{(5)}}$ is shown as a green solid line and the contour $\mathcal{L}_{C^{(5)}}$ is shown as a green dashed line.  An application of Cauchy's Integral Theorem using these contours allows elimination of the contribution of $\hat{q}^{(1)}(\cdot,t)$ and  $\hat{q}^{(2)}(\cdot,t)$ from the integral expressions~\eqref{crule}.   \label{fig:LKdV_DR_close}}
\end{figure}

Since the integrals along $\mathcal{L}_{C^{(1)}}$ and $\mathcal{L}_{C^{(3)}}$ vanish for large $C$, the integrals must vanish since the contour $\mathcal{L}_{D^{(1)}}$  becomes $\partial D_R^{(1)}$ as $C\to\infty$.  The same argument holds for $\mathcal{L}^{(3)}$ and $\partial D_R^{(3)}$. The uniform decay of the expressions in parentheses for large $k$ is exactly the condition required for the integral to vanish using Jordan's Lemma.  Hence,~\eqref{crule1} is zero.  A similar argument holds for~\eqref{crule2}.

\item[Case 2.  $\sigma_1>0, \sigma_2>0$:]

For $\sigma_1>0$ and $\sigma_2>0$, $\Gamma^{(1)}=\partial D_R^{(5)}$ and $\Gamma^{(2)}={}\partial D_R^{(1)}+\partial D_R^{(3)}$.  Using the form of $\mathcal{Y}$ in this case each piece of the integrand in~\eqref{crule1} is of the form
\begin{align*}
{}&\frac{1}{2\pi}\int_{\Gamma^{(1)}}e^{i\frac{k x}{\sigma_1}+ik^3(T-t)}\left(c_1(k) \hat{q}^{(1)}\left(\frac{\alpha^{2}k}{\sigma_1},T\right)+c_2(k) \hat{q}^{(1)}\left(\frac{\alpha k}{\sigma_1},T\right)\right)\ud k\\
&+\frac{1}{2\pi}\int_{\Gamma^{(1)}}e^{i\frac{k x}{\sigma_1}+ik^3(T-t)}c_3(k) \hat{q}^{(2)}\left(\frac{k}{\sigma_2},T\right)\ud k\\
={}&\frac{1}{2\pi}\int_{\Gamma^{(1)}} e^{ik^3(T-t)+\frac{ikx}{\sigma_1}}c_1(k)\left(\int_{-\infty}^0q^{(1)}(y,T) e^{\frac{-ik\alpha^{2}y}{\sigma_1}}\right)\ud k\ud y\\
&+\frac{1}{2\pi}\int_{\Gamma^{(1)}} e^{ik^3(T-t)+\frac{ikx}{\sigma_1}}c_2(k)\left(\int_{-\infty}^0 q^{(1)}(y,T)e^{\frac{-ik\alpha y)}{\sigma_1}}\right)\ud k\ud y\\
&+\frac{1}{2\pi}\int_{\Gamma^{(1)}} e^{ik^3(T-t)+\frac{ikx}{\sigma_1}}c_3(k)\left( \int_{0}^\infty q^{(2)}(y,T)e^{\frac{-iky}{\sigma_2}}\right)\ud k\ud y,
\end{align*}
where $c_1(k)$, $c_2(k)$, and $c_3(k)$ are $\mathcal{O}(1)$, analytic in $D^{(5)}$, and involve the constants $\beta_{j,\ell}$.  For $k\to\infty$ with $k\in\overline{D}^{(5)}$ the expression inside the parentheses  decays by the Riemann-Lebesgue Lemma.  Thus, we can apply Jordan's Lemma and Cauchy's Theorem as in Case 1, using curves shown in Figure~\ref{fig:LKdV_DR_close}.   Hence,~\eqref{crule1} is zero.  Again, a similar argument holds for~\eqref{crule2}.

\item[Case 3.  $\sigma_1>0, \sigma_2<0$:]

For $\sigma_1>0$ and $\sigma_2<0$, $\Gamma^{(1)},\Gamma^{(2)}=\partial D_R^{(5)}$.  Using the form of $\mathcal{Y}$ in this case each piece of the integrand in~\eqref{crule1} is of the form
\begin{align*}
{}&\frac{1}{2\pi}\int_{\Gamma^{(1)}} e^{i\frac{k x}{\sigma_1}+ik^3(T-t)}\left(c_1(k) \hat{q}^{(1)}\left(\frac{\alpha k}{\sigma_1},t\right)+c_2(k) \hat{q}^{(2)}\left(\frac{\alpha k}{\sigma_2},T\right)\right)\ud k\\
&+\frac{1}{2\pi}\int_{\Gamma^{(1)}} e^{i\frac{k x}{\sigma_1}+ik^3(T-t)}\left(c_3(k) \hat{q}^{(1)}\left(\frac{\alpha^2 k}{\sigma_1},T\right)+c_4(k) \hat{q}^{(2)}\left(\frac{\alpha^2 k}{\sigma_2},T\right) \right)\ud k\\
={}&\frac{1}{2\pi} \int_{\Gamma^{(1)}}e^{ik^3(T-t)+\frac{ikx}{\sigma_1}}c_1(k)\left( \int_{-\infty}^0q^{(1)}(y,T)e^{\frac{-ik\alpha y}{\sigma_1}}\right)\ud k \ud y\\
&+\frac{1}{2\pi}\int_{\Gamma^{(1)}}e^{ik^3(T-t)+\frac{ikx}{\sigma_1}}c_2(k)\left( \int_{0}^\infty q^{(2)}(y,T)e^{\frac{-ik\alpha y}{\sigma_2}}\right)\ud k \ud y\\
&+\frac{1}{2\pi}\int_{\Gamma^{(1)}}e^{ik^3(T-t)+\frac{ikx}{\sigma_1}}c_3(k)\left(\int_{-\infty}^0 q^{(1)}(y,T) e^{\frac{-ik\alpha^2 y}{\sigma_1}}\right)\ud k \ud y\\
&+\frac{1}{2\pi}\int_{\Gamma^{(1)}}e^{ik^3(T-t)+\frac{ikx}{\sigma_1}}c_4(k)\left(\int_{0}^\infty q^{(2)}(y,T) e^{\frac{-ik\alpha^2y}{\sigma_2}}\right)\ud k \ud y,
\end{align*}
where $c_1(k),c_2(k), c_3(k)$ and $c_4(k)$ are  $\mathcal{O}(1)$, involve the constants $\beta_{j,\ell}$, and are analytic for $k\in D^{(5)}$.  For $k\to\infty$ with $k\in\overline{D}^{(5)}$ the expressions inside the parentheses decays by the Riemann-Lebesgue Lemma.  As before we apply Jordan's Lemma to the appropriate curves in Figure~\ref{fig:LKdV_DR_close} and use Cauchy's Theorem.  Thus,~\eqref{crule1} is zero.  A similar argument holds for~\eqref{crule2}.

\end{description}
\end{proof}


\vspace{.4in}
{\bf Remark.}  We show in Proposition~\ref{prop:soln} that~\eqref{dsolns_2i_t} is a solution of problem~\eqref{LKdV_2i}.  It remains to show that this solution is unique in order to establish well posedness.  In an attempt to show uniqueness, we assume there exist two solutions to~\eqref{LKdV_2i}.  Let $u(x,t)$ be their difference. Then $u(x,t)$ satisfies~\eqref{LKdV_2i} with homogenous initial and interface conditions.  A standard energy argument shows 
\begin{equation}\label{wp_energy}
\begin{split}
\frac{\ud}{\ud t}\int_{-\infty}^\infty \left(\frac{\partial^n}{\partial x^n}u(x,t)\right)^2\ud x={}&\sigma_1\left(2 \frac{\partial^{n}}{\partial x^{n}} u^{(1)}(0,t) \frac{\partial^{n+2}}{\partial x^{n+2}}u^{(1)}(0,t)-\left( \frac{\partial^{n+1}}{\partial x^{n+1}}u^{(1)}(0,t)\right)^2\right)\\
&-\sigma_2\left(2 \frac{\partial^{n}}{\partial x^{n}} u^{(2)}(0,t) \frac{\partial^{n+2}}{\partial x^{n+2}}u^{(2)}(0,t)-\left( \frac{\partial^{n+1}}{\partial x^{n+1}}u^{(2)}(0,t)\right)^2\right)
\end{split}
\end{equation}
for any nonnegative integer $n$.  If the interface conditions given are such that the right-hand side of~\eqref{wp_energy} is always negative then, because $u(x,0)={}0$ and the left-hand side of~\eqref{wp_energy} is always non-negative, we have $u(x,t)\equiv0$.  Thus, one  suitable choice of interface conditions is those that satisfy this relationship.  For various signs of $\sigma_1,\sigma_2$, which we consider here, it is not clear how to establish that the solution~\eqref{dsolns_2i_t} is unique in general.  

\section{Examples}
In this section we give solutions to~\eqref{LKdV_2i} for different signs of $\sigma_1$ and $\sigma_2$ with ``canonical interface conditions."  That is, we  prescribe that the function and its first $N$ spatial derivatives are continuous across the interface where $1\leq N\leq 3$ depends on the signs of $\sigma_1$ and $\sigma_2$.
\begin{description}

\item[Example 1. $\sigma_1<0, \sigma_2>0$:]
This example requires four interface conditions.  We impose that the function, as well as its first, second, and third derivatives are continuous across the boundary.

\begin{equation*}\label{np_interfacecond}
\begin{split}
q^{(1)}(0,t)&={}q^{(2)}(0,t),\\
q^{(1)}_x(0,t)&={}q^{(2)}_x(0,t),\\
q^{(1)}_{xx}(0,t)&={}q^{(2)}_{xx}(0,t),\\
q^{(1)}_{xxx}(0,t)&={}q^{(2)}_{xxx}(0,t),\\
\end{split}
\end{equation*}
The first three conditions can be imposed directly.  The condition on the third spatial derivative can be imposed by applying the equation and integrating in $t$ to give~\eqref{3deriv_const}.  Applying the $t$ transform to~\eqref{3deriv_const} we have
\begin{equation}\label{gh_ic}
\frac{1}{\sigma_1^3}g_0(ik^3,T)-\frac{1}{\sigma_2^3}h_0(ik^3,T)=\frac{e^{ik^3T}-1}{ik^3}\left(\frac{1}{\sigma_1^3}q^{(1)}_0(0)-\frac{1}{\sigma_2^3}q^{(2)}_0(0)\right).
\end{equation}

Using elementary row operations, we have, in the notation of Proposition~\ref{prop:sigma_np}, $f_1(t)=f_2(t)=f_3(t)=0$, $f_4(t)=\frac{e^{ik^3t}-1}{ik^3}\left(\frac{1}{\sigma_1^3}q^{(1)}_0(0)-\frac{1}{\sigma_2^3}q^{(2)}_0(0)\right)$, $\beta_{25}=\beta_{36}=-1$, and the remaining $\beta_{j,\ell}={}0$.  Using these interface conditions and solving~\eqref{linsystem}, Equation~\eqref{dsolns_2i_t} becomes

\begin{equation*}
\begin{split}
q^{(1)}(x,t)={}&\frac{1}{2\pi}\int_{-\infty}^\infty e^{ikx-\omega_1t}\hat{q}^{(1)}_0(k)\ud k\\
&+\int_{\partial D_R^{(1)}}  \frac{\alpha^2\sigma_1-\sigma_2}{2\pi\alpha^2\sigma_1(\sigma_1-\sigma_2)} e^{\frac{ikx}{\sigma_1}-ik^3t}  \hat{q}^{(1)}_0\left(\frac{\alpha^2 k}{\sigma_1}\right) \ud k\\
&+\int_{\partial D_R^{(1)}} \frac{i(\alpha-1)}{2\pi\alpha^2 k\sigma_1\sigma_2(\sigma_1-\sigma_2)} (e^{-ik^3t}-1)e^{\frac{ikx}{\sigma_1}} q^{(1)}_0\left(0\right) \ud k\\
&+\int_{\partial D_R^{(1)}} \frac{\sigma_1^2(\alpha^2-1)}{2\pi\alpha^2\sigma_2^2(\sigma_1-\sigma_2)} e^{\frac{ikx}{\sigma_1}-ik^3t}  \hat{q}^{(2)}_0\left(\frac{\alpha^2 k}{\sigma_2}\right) \ud k\\
&-\int_{\partial D_R^{(1)}}  \frac{i\sigma_1^2(\alpha-1)}{2\pi\alpha^2 k\sigma_2^4(\sigma_1-\sigma_2)} (e^{-ik^3t}-1)e^{\frac{ikx}{\sigma_1}} q^{(2)}_0\left(0\right) \ud k\\
&+\int_{\partial D_R^{(3)}}  \frac{\alpha\sigma_1-\sigma_2}{2\pi k\alpha\sigma_1(\sigma_1-\sigma_2)} e^{\frac{ikx}{\sigma_1}-ik^3t}  \hat{q}^{(1)}_0\left(\frac{\alpha k}{\sigma_1}\right) \ud k\\
&+\int_{\partial D_R^{(3)}} \frac{i(\alpha^2-1)}{2\pi k(\sigma_1^3-\sigma_2^3)} (e^{-ik^3t}-1)e^{\frac{ikx}{\sigma_1}} q^{(1)}_0\left(0\right) \ud k\\
&+\int_{\partial D_R^{(3)}} \frac{\sigma_1^2(\alpha-1)}{2\pi\alpha\sigma_2^2(\sigma_1-\sigma_2)} e^{\frac{ikx}{\sigma_1}-ik^3t}  \hat{q}^{(2)}_0\left(\frac{\alpha k}{\sigma_2}\right) \ud k\\
&-\int_{\partial D_R^{(3)}}  \frac{i\sigma_1^2(\alpha^2-1)}{2\pi\sigma_2^4 \alpha k(\sigma_1-\sigma_2)} (e^{-ik^3t}-1)e^{\frac{ikx}{\sigma_1}} q^{(2)}_0\left(0\right) \ud k,
\end{split}
\end{equation*}
\begin{equation*}
\begin{split}
q^{(2)}(x,t)={}&\frac{1}{2\pi}\int_{-\infty}^\infty e^{ikx-\omega_2t}\hat{q}^{(2)}_0(k)\ud k\\
&+\int_{\partial D_R^{(1)}}  \frac{\sigma_2^2(\alpha^2-1)}{2\pi\alpha^2\sigma_1^2(\sigma_1-\sigma_2)} e^{\frac{ikx}{\sigma_2}-ik^3t}  \hat{q}^{(1)}_0\left(\frac{\alpha^2 k}{\sigma_1}\right) \ud k\\
&+\int_{\partial D_R^{(1)}} \frac{i(\alpha-1)(\sigma_2-\alpha\sigma_1+\alpha\sigma_2)}{2\pi\alpha^2 k\sigma_1^3(\sigma_1-\sigma_2)} (e^{-ik^3t}-1)e^{\frac{ikx}{\sigma_2}} q^{(1)}_0\left(0\right) \ud k\\
&+\int_{\partial D_R^{(1)}} \frac{\alpha^2\sigma_2-\sigma_1}{2\pi\alpha^2\sigma_2(\sigma_1-\sigma_2)} e^{\frac{ikx}{\sigma_2}-ik^3t}  \hat{q}^{(2)}_0\left(\frac{\alpha^2 k}{\sigma_2}\right) \ud k\\
&-\int_{\partial D_R^{(1)}}  \frac{i(\alpha-1)(\alpha\sigma_1-\alpha\sigma_2-\sigma_2)}{2\pi\alpha^2 k\sigma_2^3(\sigma_1-\sigma_2)} (e^{-ik^3t}-1)e^{\frac{ikx}{\sigma_2}} q^{(2)}_0\left(0\right) \ud k\\
&+\int_{\partial D_R^{(3)}}  \frac{\sigma_2^2(\alpha-1)}{2\pi\alpha \sigma_1^2(\sigma_1-\sigma_2)} e^{\frac{ikx}{\sigma_2}-ik^3t}  \hat{q}^{(1)}_0\left(\frac{\alpha k}{\sigma_1}\right) \ud k\\
&-\int_{\partial D_R^{(3)}}  \frac{i(\alpha-1)(\sigma_2+\alpha\sigma_1)}{2\pi\alpha \sigma_1^3k(\sigma_1-\sigma_2)} (e^{-ik^3t}-1)e^{\frac{ikx}{\sigma_2}} q^{(1)}_0\left(0\right) \ud k\\
&+\int_{\partial D_R^{(3)}} \frac{\alpha\sigma_2-\sigma_1}{2\pi\alpha\sigma_2(\sigma_1-\sigma_2)} e^{\frac{ikx}{\sigma_2}-ik^3t}  \hat{q}^{(2)}_0\left(\frac{\alpha k}{\sigma_2}\right) \ud k\\
&-\int_{\partial D_R^{(3)}}  \frac{i(\alpha-1)(\sigma_2+\alpha\sigma_1)}{2\pi\alpha k\sigma_2^3(\sigma_1-\sigma_2)} (e^{-ik^3t}-1)e^{\frac{ikx}{\sigma_2}} q^{(2)}_0\left(0\right) \ud k.
\end{split}
\end{equation*}

\vspace{.4in}
{\bf Remark.}
Combining~\eqref{3deriv_const} with the first interface condition of Example 1 and the mutual compatibility of $q_0^{(1)}(x)$ and $q_0^{(2)}(x)$ with respect to the interface conditions, we find that $q(0,t)$ is constant in time. That is, for all $t\geq0$, $q^{(1)}(0,t) = q_0^{(1)}(0) = q_0^{(2)}(0) = q_0^{(2)}(0,t)$. Note that this property of the interface problem is obtained without any reference to a solution method or any appeal to the first or second order interface conditions. The compatibility condition $q_0^{(1)}(0) = q_0^{(2)}(0)$ could also be used to simplify~\eqref{gh_ic}. We avoid this simplification because even without a smooth initial condition the formulae obtained in this paper are valid for $t>0$ and the short time asymptotics can be analyzed as in~\cite{BiondiniTrogdon}.

\item[Example 2. $\sigma_1>0, \sigma_2>0$:]
This example requires three interface conditions.  We impose that the function, as well as its first and second derivative are continuous across the boundary.  That is,
\begin{equation}\label{pp_interfacecond}
\begin{split}
q^{(1)}(0,t)&={}q^{(2)}(0,t),\\
q^{(1)}_{x}(0,t)&={}q^{(2)}_{x}(0,t),\\
q^{(1)}_{xx}(0,t)&={}q^{(2)}_{xx}(0,t).
\end{split}
\end{equation}

In the notation of Proposition~\ref{prop:sigma_pp} $f_1(t)={}f_2(t)={}f_3(t)={}0$, $\beta_{11}={}\beta_{22}={}\beta_{33}={}-1$, and the remaining $\beta_{j,\ell}={}0$.  Using the interface conditions~\eqref{pp_interfacecond} and solving~\eqref{linsystem}, Equation~\eqref{dsolns_2i_t} becomes

\begin{equation*}
\begin{split}
q^{(1)}(x,t)={}&\frac{1}{2\pi}\int_{-\infty}^\infty e^{ikx-\omega_1t}\hat{q}^{(1)}_0(k)\ud k\\
&+\int_{\partial D_R^{(5)}}  \frac{(\sigma_1-\sigma_2)(\sigma_1+\alpha\sigma_1+\alpha\sigma_2)}{2\pi\alpha\sigma_1(\sigma_1-\alpha \sigma_2)(\sigma_1+\sigma_2+\alpha\sigma_2)}  e^{\frac{ikx}{\sigma_1}-ik^3t}  \hat{q}_0^{(1)}\left(\frac{\alpha k}{\sigma_1}\right)  \ud k,\\
&+\int_{\partial D_R^{(5)}}\frac{\sigma_2-\sigma_1}{2\pi\alpha\sigma_1(\sigma_1+\sigma_2+\alpha\sigma_2)}  e^{\frac{ikx}{\sigma_1}-ik^3t}    \hat{q}_0^{(1)}\left(\frac{\alpha^2 k}{\sigma_1}\right)  \ud k,\\
&-\int_{\partial D_R^{(5)}} \frac{3\sigma_1^3}{2\pi\sigma_2(\sigma_1-\alpha \sigma_2)(\sigma_1+\sigma_2+\alpha\sigma_2)}  e^{\frac{ikx}{\sigma_1}-ik^3t}   \hat{q}_0^{(2)}\left(\frac{ k}{\sigma_2}\right)  \ud k,
\end{split}
\end{equation*}
\begin{equation*}
\begin{split}
q^{(2)}(x,t)={}&\frac{1}{2\pi}\int_{-\infty}^\infty e^{ikx-\omega_2t}\hat{q}^{(2)}_0(k)\ud k\\
&-\int_{\partial D_R^{(1)}} \frac{\sigma_2(\sigma_1^2+\sigma_1\sigma_2-\sigma_2^2)}{2\pi\sigma_1^2(\sigma_1^2+\alpha(1+\alpha)\sigma_1\sigma_2-\sigma_2^2)} e^{\frac{ikx}{\sigma_2}-ik^3t} \hat{q}_0^{(1)}\left(\frac{k}{\sigma_1}\right) \ud k\\
&+\int_{\partial D_R^{(1)}}  \frac{\sigma_2(\sigma_2(\sigma_1+\sigma_2)+\alpha(\sigma_1^2+\sigma_2^2))}{2\pi\alpha\sigma_1^2(\sigma_1^2+\alpha(1+\alpha)\sigma_1\sigma_2-\sigma_2^2)} e^{\frac{ikx}{\sigma_2}-ik^3t} \hat{q}_0^{(1)}\left(\frac{\alpha k}{\sigma_1}\right) \ud k\\
&-\int_{\partial D_R^{(1)}}  \frac{\sigma_1^2+(1+\alpha)\sigma_1\sigma_2-\alpha\sigma_2^2}{2\pi\alpha\sigma_2(\sigma_1^2+\alpha(1+\alpha)\sigma_1\sigma_2-\sigma_2^2)} e^{\frac{ikx}{\sigma_2}-ik^3t} \hat{q}_0^{(2)}\left(\frac{\alpha^2 k}{\sigma_2}\right) \ud k\\
&+\int_{\partial D_R^{(3)}} \frac{\sigma_2(\sigma_1^2+\sigma_1\sigma_2-\sigma_2^2)}{2\pi\sigma_1^2(\alpha\sigma_1^2(1+\alpha)+\sigma_2(\sigma_1+\sigma_2))} e^{\frac{ikx}{\sigma_2}-ik^3t}\hat{q}_0^{(1)}\left(\frac{k}{\sigma_1}\right) \ud k\\
&+\int_{\partial D_R^{(3)}}  \frac{\sigma_2(\alpha \sigma_1(\sigma_2-\sigma_1)+\sigma_2(\sigma_1+\sigma_2))}{2\pi\alpha\sigma_1^2(\alpha\sigma_1^2(1+\alpha)+\sigma_2(\sigma_1+\sigma_2))} e^{\frac{ikx}{\sigma_2}-ik^3t} \hat{q}_0^{(1)}\left(\frac{\alpha^2 k}{\sigma_1}\right) \ud k\\
&-\int_{\partial D_R^{(3)}} \frac{(1+\alpha)\sigma_1^2+\sigma_1\sigma_2+\alpha \sigma_2^2}{2\pi\alpha\sigma_2(\alpha\sigma_1^2(1+\alpha)+\sigma_2(\sigma_1+\sigma_2))} e^{\frac{ikx}{\sigma_2}-ik^3t}  \hat{q}_0^{(2)}\left(\frac{\alpha k}{\sigma_2}\right) \ud k.
\end{split}
\end{equation*}

\item[Example 3. $\sigma_1>0, \sigma_2<0$:]
This example requires two interface conditions.  We impose that the function and its first derivative are continuous across the boundary.  That is,
\begin{equation}\label{pn_interfacecond}
\begin{split}q^{(1)}(0,t)&={}q^{(2)}(0,t),\\
q^{(1)}_x(0,t)&={}q^{(2)}_x(0,t).
\end{split}
\end{equation}

In the notation of Proposition~\ref{prop:sigma_pn} $f_1(t)={}f_2(t)={}0$, $\beta_{15}={}\beta_{25}={}-1$, and the remaining $\beta_{j,\ell}={}0$.  Using the interface conditions~\eqref{pn_interfacecond} and solving~\eqref{linsystem}, Equation~\eqref{dsolns_2i_t} becomes

\begin{equation*}
\begin{split}
q^{(1)}(x,t)={}&\frac{1}{2\pi}\int_{-\infty}^\infty e^{ikx-\omega_1t}\hat{q}^{(1)}_0(k)\ud k+\int_{\partial D_R^{(5)}} \frac{\sigma_1+\alpha\sigma_1-\sigma_2}{2\pi\alpha\sigma_1(\sigma_1+\sigma_2)}e^{\frac{ikx}{\sigma_1}-ik^3t}  \hat{q}_0^{(1)}\left(\frac{\alpha k}{\sigma_1}\right)  \ud k\\
&+\int_{\partial D_R^{(5)}}  \frac{\sigma_2+\alpha\sigma_2-\sigma_1}{2\pi\alpha\sigma_1(\sigma_1+\sigma_2)}e^{\frac{ikx}{\sigma_1}-ik^3t} \hat{q}_0^{(1)}\left(\frac{\alpha^2 k}{\sigma_1}\right)  \ud k\\
&+\int_{\partial D_R^{(5)}}  \frac{\sigma_1(2+\alpha)}{2\pi\alpha\sigma_2(\sigma_1+\sigma_2)}e^{\frac{ikx}{\sigma_1}-ik^3t} \hat{q}_0^{(2)}\left(\frac{\alpha k}{\sigma_2}\right)  \ud k\\
&-\int_{\partial D_R^{(5)}} \frac{\sigma_1(2+\alpha)}{2\pi\alpha\sigma_2(\sigma_1+\sigma_2)}e^{\frac{ikx}{\sigma_1}-ik^3t}  \hat{q}_0^{(2)}\left(\frac{\alpha^2 k}{\sigma_2}\right)  \ud k,
\end{split}
\end{equation*}
\begin{equation*}
\begin{split}
q^{(2)}(x,t)={}&\frac{1}{2\pi}\int_{-\infty}^\infty e^{ikx-\omega_2t}\hat{q}^{(2)}_0(k)\ud k-\int_{\partial D_R^{(5)}} \frac{\sigma_2(2+\alpha)}{2\pi\alpha\sigma_1(\sigma_1+\sigma_2)} e^{\frac{ikx}{\sigma_2}-ik^3t}  \hat{q}_0^{(1)}\left(\frac{\alpha k}{\sigma_1}\right)  \ud k\\
&+\int_{\partial D_R^{(5)}} \frac{\sigma_2(2+\alpha)}{2\pi\alpha\sigma_1(\sigma_1+\sigma_2)} e^{\frac{ikx}{\sigma_2}-ik^3t} \hat{q}_0^{(1)}\left(\frac{\alpha^2 k}{\sigma_1}\right)  \ud k\\
&-\int_{\partial D_R^{(5)}}  \frac{\sigma_2+\alpha\sigma_2-\sigma_1}{2\pi\alpha\sigma_2(\sigma_1+\sigma_2)} e^{\frac{ikx}{\sigma_2}-ik^3t} \hat{q}_0^{(2)}\left(\frac{\alpha k}{\sigma_2}\right)  \ud k\\
&-\int_{\partial D_R^{(5)}}  \frac{\sigma_1+\alpha\sigma_1-\sigma_2}{2\pi\alpha\sigma_2(\sigma_1+\sigma_2)} e^{\frac{ikx}{\sigma_2}-ik^3t}  \hat{q}_0^{(2)}\left(\frac{\alpha^2 k}{\sigma_2}\right)  \ud k.
\end{split}
\end{equation*}

\end{description}

\section*{Acknowledgements}
N.E.S. acknowledges support from the National Science Foundation under grant number NSF-DGE-0718124.  Any opinions, findings, and conclusions or recommendations expressed in this material are those of the authors and do not necessarily reflect the views of the funding sources.
\bibliographystyle{plain}
\bibliography{./FullBib}

\end{document}